\documentclass[12pt]{article}
\usepackage{amsmath,amsfonts,amssymb,amsthm}
\usepackage{bbm,mathrsfs,bm,mathtools}
\title{Vectorial analogues\\ of Cauchy's surface area formula}
\author{Daniel Hug and Rolf Schneider}
\date{}
\newcommand{\R}{{\mathbb R}}

\newcommand{\K}{{\mathcal K}}

\newcommand{\N}{{\mathbb N}}

\newcommand{\Ha}{\mathcal{H}}

\newcommand{\B}{\mathcal{B}}
\newcommand{\D}{{\rm d}}

  \newcommand{\inn}{{\rm int}\,}

  \renewcommand{\dim}{{\rm dim}\,}
  
  \newcommand{\relint}{{\rm relint}\,}

  \newcommand{\relbd}{{\rm relbd}\,}

\newtheorem{theorem}{Theorem}
\newtheorem{lemma}{Lemma}
\newtheorem{corollary}{Corollary}

\begin{document}
\maketitle

\begin{abstract}
Cauchy's surface area formula says that for a convex body $K$ in $n$-dimensional Euclidean space the mean value of the $(n-1)$-dimensional volumes of the orthogonal projections of $K$ to hyperplanes is a constant multiple of the surface area of $K$. We prove an analogous formula, with the volumes of the projections replaced by their moment vectors. This requires to introduce a new vector-valued valuation on convex bodies.
\\[2mm]
\end{abstract}

\section{Introduction}\label{sec1}

Cauchy's surface area formula is probably the first result in a branch later named integral geometry. The $n$-dimensional version of Cauchy's formula says that
$$ S(K) = \frac{1}{\kappa_{n-1}} \int_{S^{n-1}} V_{n-1}(K|u^\perp)\,\D u$$
for each convex body $K$ (a nonempty, compact, convex set) in Euclidean space $\R^n$. Here $K|u^\perp$ is the image of $K$ under orthogonal projection to the linear subspace $u^\perp$ orthogonal to the unit vector $u$. By $S(K)$ we denote the surface area of $K$ and by $V_k$ the $k$-dimensional volume. Further, $\int_{S^{n-1}}f(u)\, \D u$ always denotes integration over the unit sphere $S^{n-1}$ with respect to the spherical Lebesgue measure, and $\kappa_k$ is the volume of the $k$-dimensional unit ball ($\omega_k=k\kappa_k$ is its surface area). For proofs of more general projection formulas we refer to \cite[p. 301]{Sch14} and \cite[p. 222]{SW08}  or \cite[Chap.~5]{HW20}.

Classical integral geometry, as it can be found, e.g., in the books by Hadwiger \cite[Chap. 6]{Had57} and Santal\'o \cite{San76} (see also \cite[Sect. 4.4]{Sch14}), exhibits, for example, intersection theorems such as the kinematic formulas and the Crofton formulas, mainly for real-valued valuations on convex bodies. Integral geometric formulas for vector-valued functionals were investigated in \cite{HS71, Mul53, Sch72}. An extensive theory of integral geometric intersection formulas for tensor valued valuations, also with values in spaces of measures, was developed in \cite{BH18, HSS08, HW17, HW18a, HW18b, Sch00, SS02}. We also refer to the survey article \cite{BH17}.

While the classical integral geometry of real-valued functionals also considered projections, the mentioned investigations of tensor valuations are restricted to intersections. It seems that not even such a simple question as the following vector-valued analogue of Cauchy's surface area formula has been answered. First we recall (e.g., from \cite[Sect. 5.4.1]{Sch14}) that for a convex body $K\subset\R^n$ and for $k\in\{\dim K,\dots,n\}$ the moment vector $z_{k+1}(K)$ is defined by
$$ z_{k+1}(K):= \int_K x\,\Ha^k(\D x),$$
where $\Ha^k$ is the $k$-dimensional Hausdorff measure (the lower index indicates the degree of homogeneity). The question is then whether the integral
$$ \int_{S^{n-1}} z_n(K|u^\perp)\,\D u$$
can be expressed by familiar functionals of $K$. It turns out that, compared to the cases dealing with intersections, additional tensor valuations are required. After introducing these, we shall give an answer to the question in Theorem \ref{T2.1}.

\section{Preliminaries, and a result}\label{sec2}

We work in $n$-dimensional Euclidean vector space $\R^n$ (with origin $o$) and use its scalar product $\langle\cdot\,,\cdot\rangle$ to identify the space with its dual space. Thus, we identify a vector $x\in\R^n$ with the linear functional $\langle x,\cdot\rangle$. By ${\mathbb T}^r$ we denote the real vector space (with its usual topology) of symmetric $r$-tensors (tensors of rank $r\in\N_0$) on $\R^n$. The elements of ${\mathbb T}^r$ are symmetric $r$-linear functionals on $\R^n$. The symmetric tensor product of $a\in{\mathbb T}^r$ and $b\in{\mathbb T}^s$ is denoted by $a\cdot b = ab$ and is an element of ${\mathbb T}^{r+s}$. In particular, for $x\in\R^n$ we write $x^r:= x\cdots x$ ($r$ factors); we then have $x^r(a_1,\dots,a_r)=\langle x,a_1\rangle \cdots \langle x,a_r\rangle$ for $a_1,\dots,a_r\in\R^n$.

Let $K\in\K^n$, where $\K^n$ denotes the set of convex bodies in $\R^n$. Its volume can be written as
\begin{equation}\label{2.1}
V_n(K)= \int_K\D x,
\end{equation}
where $\int f(x)\,\D x$ indicates integration with respect to the Lebesgue measure on $\R^n$. This motivates one to define a symmetric $r$-tensor by
$$ \Psi_r(K):= \frac{1}{r!} \int_Kx^r\D x,\quad K\in\K^n.$$
The factor before the integral is for convenience, and we have $\Psi_1(K) = z_{n+1}(K)$. For $z_{n+1}$, we get the polynomial expansion
$$ z_{n+1}(K+\lambda B^n) = \sum_{j=0}^n\binom{n}{j} \lambda^j q_j(K)$$
(see \cite[(5.95)]{Sch14}), where $B^n$ is the unit ball of $\R^n$ and $\lambda\ge 0$. The vector-valued coefficients $q_j$ can be represented by
$$ q_j(K) = \frac{1}{n} \int_{\partial K} x\,C_{n-j}(K,\D x)$$
for $j\ge 1$ (whereas $q_0=z_{n+1}$). See \cite[(5.98)]{Sch14}, and note that the curvature measure $C_m(K,\cdot)$ is concentrated on the boundary of $K$. In particular (by \cite[(4.31)]{Sch14}),  if $K$ has nonempty interior, then
\begin{equation}\label{2.2}
q_1(K) = \frac{1}{n} \int_{\partial K} x\,\Ha^{n-1}(\D x).
\end{equation}

Let $\K^n_o$ denote the set of convex bodies $K\in\K^n$ with $o\in \inn K$. Alternatively to (\ref{2.1}), for $K\in\K^n_o$ the volume of $K$ can also be written as the integral of the cone-volume measure of $K$ over the unit sphere $S^{n-1}$.  The cone-volume measure of $K\in\K^n_o$ is defined by
$$ V_K(\omega) := \Ha^n\left(\bigcup_{x\in\tau(K,\omega)} [o,x]\right),\quad \omega\in \B(S^{n-1}).$$
Here $\tau(K,\omega)$ is the set of boundary points of $K$ at which there exists an outer unit normal vector of $K$ belonging to $\omega$, $[o,x]$ is the closed segment with endpoints $o$ and $x$, and $\B(S^{n-1})$ denotes the $\sigma$-algebra of Borel sets in $S^{n-1}$.
It is known (see, e.g., \cite[p. 501]{Sch14}) that
\begin{align*}
 V_K(\omega)&= \frac{1}{n} \int_\omega h_K(u)\, S_{n-1}(K,\D u)\\
 &=\frac{1}{n}\int_{\R^n\times \omega} \langle x,u\rangle  \, \Theta_{n-1}(K,\D (x,u)),
 \end{align*}
where $h_K$ is the support function of $K$, $S_{n-1}(K,\cdot)$ is its surface area measure and $\Theta_{n-1}(K,\cdot)$ is a support measure of $K$ (see \cite[Sect.~4.2]{Sch14}). For the last equality, we refer to \cite[(4.11)]{Sch14} and the simple observation that $h_K(u)=\langle x,u\rangle$ for $(x,u)$ in the support of the measure $\Theta_{n-1}(K,\cdot)$. From these representations it is clear that $V_K$ is a measure-valued valuation (see \cite[Thm. 4.2.1]{Sch14}). Moreover, the integral representations are well defined for all convex bodies.

The preceding discussion suggests to define a symmetric $r$-tensor functional by
$$ \Upsilon_r(K) := \int_{S^{n-1}} u^r\, V_K(\D u),\quad K\in\K^d_o.$$
For general convex bodies $K\in\K^d$, it is consistent to define $\Upsilon_r:\K^n\to{\mathbb T}^r$ by
\begin{align}\label{ups}
\Upsilon_r(K):&= \frac{1}{n}\int_{S^{n-1}} h_K(u)u^r\,S_{n-1}(K,\D u)\\
&=\frac{1}{n}\int_{\R^n\times S^{n-1}} \langle x,u\rangle u^r \, \Theta_{n-1}(K,\D (x,u)).\nonumber
\end{align}

Just as the $\Psi_r$, also the $\Upsilon_r$ are rotation covariant, continuous tensor-valued valuations. But the translation behavior is different. We know (see \cite[(5.104)]{Sch14}) that, for $t\in\R^n$,
$$ \Psi_r(K+t) = \sum_{j=0}^r \frac{1}{j!} \Psi_{r-j}(K)t^j,$$
where $\Psi_{r-j}(K)t^j$ is a symmetric tensor product. On the other hand, if we define a translation invariant $r$-tensor by
$$ \Xi_r(K):= \frac{1}{n} \int_{S^{n-1}} u^r\,S_{n-1}(K,\D u),$$
for $K\in\K^n$, then
$$ \Upsilon_r(K+t) = \Upsilon_r(K) + \Xi_{r+1}(K)(t),$$
where the symmetric $(r+1)$-tensor $\Xi_{r+1}(K)$ is applied to the vector $t$, which results in a tensor of rank $r$.

Considering that $\Upsilon_0(K)$ is the volume of $K$ and that
$$ \frac{1}{n}\int_{S^{n-1}} h_{K_1}(u)\, S(K_2,\dots,K_n,\D u)  = V(K_1,K_2,\dots,K_n)$$
(where $S(K_2,\dots,K_n,\cdot)$ is the mixed area measure of $K_2,\dots,K_n$) is the mixed volume of $K_1,\dots,K_n$ and hence is symmetric in its arguments, one may ask whether also
\begin{equation}\label{eqnosym}
 \int_{S^{n-1}} h_{K_1}(u)u^r S(K_2,\dots,K_n,\D u)
 \end{equation}
is symmetric in $K_1,\dots,K_n$. That this is in general not the case, can be seen from the fact that, for $r=1$, the vector
$$ \int_{S^{n-1}} h_K(u)u\, S(B^n,\dots,B^n,\D u)$$
is proportional to the Steiner point of $K$ (as follows from \cite[(1.31)]{Sch14}), whereas
$$ \int_{S^{n-1}} h_{B^n}(u)u \,S(K,B^n,\dots,B^n,\D u)$$
is always equal to $o$, by \cite[(5.30)]{Sch14}.

Since \eqref{eqnosym} is not symmetric in $K_1,\ldots,K_n$, we introduce a symmetric tensor functional $\Upsilon^{(r)}:(\K^n)^n\to {\mathbb T}^r$ by
$$
\Upsilon^{(r)}(K_1,\ldots,K_n):=\frac{1}{n}\sum_{i=1}^n\frac{1}{n}  \int_{S^{n-1}} h_{K_i}(u)u^r \,S(K_1,\dots,\check{K_i},\ldots ,K_n,\D u),
$$
where $\check{K_i}$ means that $K_i$ is omitted.

We have mentioned tensor valuations of rank higher than necessary for the following, since we hope to come back to them later.

We can now state the vectorial counterpart to Cauchy's surface area formula.
\begin{theorem}\label{T2.1}
For $K\in\K^n$,
$$ \int_{S^{n-1}} z_n(K|u^\perp)\,\D u = \frac{n\kappa_{n-1}}{n+1}(nq_1(K)-\Upsilon_1(K)).$$
\end{theorem}

The functional $q_1$ is a special mixed moment vector, that is, we have
$$\frac{n}{n+1}q_1(K)=z(K[n],B^n).$$
The mixed moment vectors $z:(\K^n)^{n+1}\to\R^n$ are symmetric and multilinear in each of their $n+1$ arguments (see \cite[Section 5.4.1]{Sch14}. The choice $K=\lambda_1K_1+\dots+\lambda_nK_n$ with $\lambda_1,\ldots,\lambda_n\ge 0$ in Theorem \ref{T2.1}, polynomial expansion and comparison of the coefficients of  $\lambda_1\cdots\lambda_n$ yields the following consequence.

\begin{corollary}\label{C2.2}
Let $K_1,\ldots,K_n\in\K^n$. Then
\begin{align*}
 &\int_{S^{n-1}} z_n(K_1|u^\perp,\ldots,K_n|u^\perp)\,\D u\\
 &=n\kappa_{n-1}\left(z(K_1,\ldots,K_n,B^n)-\frac{1}{n+1}\Upsilon^{(1)}(K_1,\ldots,K_n)\right).
\end{align*}
\end{corollary}

\section{Proof of Theorem \ref{T2.1}}\label{sec3}

Both sides of the asserted equation depend continuously on $K$. Hence we can assume that  $K\in\K^n$ has nonempty interior.
Let $\partial'K$ denote the set of all $x\in\partial K$ at which the outer unit normal vector $\nu_K(x)$ of $K$ at $x$ is unique. Then $\partial'K$ is a Borel set, $\nu_K$ is continuous on $\partial'K$ and $\partial K\setminus \partial'K$ has $\Ha^{n-1}$-measure zero.
  Moreover, for $y\in \partial'K$ and $u\in S^{n-1}$ the projection map $p_u:\partial K\to u^\perp$, $y\mapsto y-\langle y,u\rangle u$, has the Jacobian $Jp_u(y)=|\langle \nu_K(y),u\rangle|$. For $x\in \relint(K|u^\perp)$ we have $p_u^{-1}(\{x\})=\{y_u^-(K,x),y_u^+(K,x)\}$ and $p_u(y^\pm_u(K,x))=x$.
 Therefore, we can apply the coarea formula and then Fubini's theorem to obtain
\begin{eqnarray}
2\int_{S^{n-1}} z_n(K|u^\perp)\,\D u &=& 2\int_{S^{n-1}} \int_{K|u^\perp} x\,\Ha^{n-1}(\D x)\,\D u\nonumber\\
&=&\int_{S^{n-1}}\int_{K|u^\perp}\int_{p_u^{-1}(\{x\})}p_u(y)\, \mathcal{H}^0(\D y)\, \mathcal{H}^{n-1}(\D x)\,\D u\nonumber\nonumber\\
&=& \int_{S^{n-1}} \int_{\partial K} |\langle \nu_K(y),u\rangle|(y-\langle y,u\rangle u)\,\Ha^{n-1}(\D y)\,\D u\nonumber\\
&=&\int_{\partial K} y \int_{S^{n-1}}  |\langle \nu_K(y),u\rangle|\,\D u\,\Ha^{n-1}(\D y)\nonumber\\
&&  - \int_{\partial K} \int_{S^{n-1}} |\langle  \nu_K(y),u\rangle|\langle y,u\rangle u\,\D u\,\Ha^{n-1}(\D y).\label{3.1}
\end{eqnarray}
For the first integral in (\ref{3.1}) we can use that
$$ \int_{S^{n-1}} |\langle v,u\rangle|\,\D u= 2\kappa_{n-1}$$
for all $v\in S^{n-1}$ and obtain
$$ \int_{\partial K} y \int_{S^{n-1}}  |\langle \nu_K(y),u\rangle|\,\D u\,\Ha^{n-1}(\D y) = 2n\kappa_{n-1} q_1(K),$$
by (\ref{2.2}). For the second integral, we need a lemma. For this, we denote by $Q$ the metric $2$-tensor on $\R^n$, that is, $Q(x,y)=\langle x,y\rangle$ for $x,y\in\R^n$.

\begin{lemma}\label{L3.1}
Let $n\ge 2$ and $v\in S^{n-1}$. Then
$$ \int_{S^{n-1}} |\langle v,u\rangle|u^2\,\D u= \frac{2\kappa_{n-1}}{n+1}(v^2+Q).$$
\end{lemma}

\begin{proof}
We use a decomposition of spherical Lebesgue measure to get
\begin{eqnarray}
&& \int_{S^{n-1}} |\langle v,u\rangle|u^2\,\Ha^{n-1}(\D u)\nonumber\\
&& = \int_{-1}^1 \int_{S^{n-1}\cap v^\perp}(1-\tau^2)^{\frac{n-3}{2}}\left|\left\langle v, \tau v+\sqrt{1-\tau^2}\,w\right\rangle\right|\left(\tau v+\sqrt{1-\tau^2}\,w\right)^2\Ha^{n-2}(\D w)\,\D \tau\nonumber\\
&& = \int_{-1}^1 \int_{S^{n-1}\cap v^\perp}(1-\tau^2)^{\frac{n-3}{2}}|\tau|\left(\tau^2v^2+2\tau\sqrt{1-\tau^2}\,vw+(1-\tau^2)w^2\right)\Ha^{n-2}(\D w)\,\D \tau\nonumber\\
&&= \int_{-1}^1 |\tau|\tau^2(1-\tau^2)^{\frac{n-3}{2}}\D\tau\,\omega_{n-1}v^2 + \int_{-1}^1|\tau|(1-\tau^2)^{\frac{n-1}{2}}\D\tau \int_{S^{n-1}\cap v^\perp} w^2\,\Ha^{n-2}(\D w), \label{3.2}
\end{eqnarray}
since
$$ \int_{S^{n-1}\cap v^\perp} w\,\Ha^{n-2}(\D w)=o.$$
Let $B(\cdot\,,\cdot)$ denote the Beta function. Then
\begin{eqnarray}
 \int_{-1}^1 |\tau|\tau^2(1-\tau^2)^{\frac{n-3}{2}}\D\tau &=& 2\int_0^1 \tau^3(1-\tau^2)^{\frac{n-3}{2}}\D\tau = \int_0^1 s(1-s)^{\frac{n-3}{2}}\D s\nonumber\\
&=& B\left(2,\frac{n-1}{2}\right) =\frac{\Gamma(2)\Gamma(\frac{n-1}{2})}{\Gamma(\frac{n+3}{2})} =\frac{4}{n^2-1}\label{3.3}
\end{eqnarray}
and
$$ \int_{-1}^1|\tau|(1-\tau^2)^{\frac{n-1}{2}}\D\tau = \int_0^1(1-s)^{\frac{n-1}{2}}\D s = \frac{2}{n+1}.$$
It is known from \cite[(24)]{SS02} that
\begin{equation}\label{3.4}
\int_{S^{n-1}\cap v^\perp} w^2\,\Ha^{n-2}(\D w) = 2\frac{\omega_{n+1}}{\omega_3}Q_{v^\perp} =\kappa_{n-1}(Q-v^2),
\end{equation}
where $Q_{v^\perp} = Q-v^2$ is the metric tensor of $v^\perp$. Combination of (\ref{3.2})--(\ref{3.4}) yields the assertion.
\end{proof}

Using Lemma \ref{L3.1}, we can write the second integral in (\ref{3.1}) as
\begin{eqnarray*}
&& \int_{\partial K} \int_{S^{n-1}} |\langle u,\nu_K(y)\rangle|\langle y,u\rangle u\,\D u\,\Ha^{n-1}(\D y)\allowdisplaybreaks\\
&& = \int_{\partial K} \int_{S^{n-1}} |\langle u,\nu_K(y)\rangle|u^2(y)\,\D u\,\Ha^{n-1}(\D y)\allowdisplaybreaks\\
&& = \frac{2\kappa_{n-1}}{n+1} \int_{\partial K} (\nu_K(y)^2+Q)(y)\,\Ha^{n-1}(\D y)\\
&& = \frac{2\kappa_{n-1}}{n+1} \int_{\partial K} \langle\nu_K(y),y\rangle\nu_K(y)\,\Ha^{n-1}(\D y) + \frac{2\kappa_{n-1}}{n+1} \int_{\partial K} y\,\Ha^{n-1}(\D y)\\
&& =  \frac{2\kappa_{n-1}}{n+1} \int_{S^{n-1}} h_K(u)u\,S_{n-1}(K,\D u) +  \frac{2n\kappa_{n-1}}{n+1}q_1(K)\\
&& = \frac{2n\kappa_{n-1}}{n+1}\left[\Upsilon_1(K) + q_1(K)\right].
\end{eqnarray*}
We have used that the surface area measure $S_{n-1}(K,\cdot)$ is the image measure of $\Ha^{n-1}$ under the Gauss map, together with the transformation theorem for integrals, and we have again used (\ref{2.2}), as well as (\ref{ups}). Taking both integrals of (\ref{3.1}) together, we complete the proof of Theorem \ref{T2.1}.

\section{Another analogue of Cauchy's formula}\label{sec4}

We mention briefly another way of obtaining an analogue of Cauchy's surface area formula. For this, we consider a convex body $K\in\K^n_o$ and a measurable, bounded function $f:\partial K\to {\mathbb T}^r$. For $x\in K|u^\perp$ the intersection $K\cap (x+\R u)$ is a segment $[y_u^-(K,x),y_u^+(K,x)]$ with boundary points $y_u^-(K,x),y_u^+(K,x)\in\partial K$ satisfying, say, $\langle y_u^-(K,x),u\rangle\le\langle y_u^+(K,x),u\rangle$. Defining the $r$-tensor
$$ F_f(K,u):= \int_{K|u^\perp} f(y_u^+(K,x))\,\Ha^{n-1}(\D x),$$
we state that
\begin{equation}\label{4.1}
\int_{S^{n-1}} F_f(K,u)\,\D u= \kappa_{n-1} \int_{\partial K} f(x)\,\Ha^{n-1}(\D x).
\end{equation}
For $r=0$ and $f=1$ this is Cauchy's surface area formula. While \eqref{4.1} involves  integration over projections, the relation is rather a  statement about real-valued than tensor-valued functions.

For the proof of (\ref{4.1}), we write
\begin{eqnarray*}
\partial_u^\pm K &:=& \{y_u^\pm(K,x): x\in K|u^\perp\}.
\end{eqnarray*}
By the coarea formula we then have
\begin{equation}\label{eqcoa}
\int_{\partial_u^\pm K} f(x)|\langle \nu_K(x),u\rangle|\,\Ha^{n-1}(\D x) 
=\int_{K|u^\perp} f(y_u^\pm(K,x))\,\Ha^{n-1}(\D x).
\end{equation}
Since  $\langle\nu_K(x),u\rangle=0$ for $\mathcal{H}^{n-1}$-almost all $x\in K\cap (\relbd (K|u^\perp) +\R u)$ (here $\relbd$ refers to the boundary relative to $u^\perp$), we have
\begin{equation}\label{eqzero}
 \int_{K\cap (\relbd (K|u^\perp) +\R u) } f(x)|\langle\nu_K(x),u\rangle|\,\Ha^{n-1}(\D x)=0.
\end{equation}
It follows from $y_{-u}^+(K,x)= y_u^-(K,x)$, \eqref{eqcoa} and \eqref{eqzero} that
$$ F_f(K,u) +F_f(K,-u)=   \int_{\partial K} f(x)|\langle\nu_K(x),u\rangle|\,\Ha^{n-1}(\D x).$$
Fubini's theorem gives
\begin{eqnarray*}
\int_{S^{n-1}} F_f(K,u)\,\D u &=& \frac{1}{2} \int_{S^{n-1}} \int_{\partial K} f(x)|\langle \nu_K(x),u\rangle|\,\Ha^{n-1}(\D x)\,\D u\\
&=& \frac{1}{2} \int_{\partial K} f(x)\int_{S^{n-1}}|\langle \nu_K(x),u\rangle|\,\D u\,\Ha^{n-1}(\D x)\\
&=& \kappa_{n-1} \int_{\partial K} f(x)\,\Ha^{n-1}(\D x).
\end{eqnarray*}
This completes the proof of (\ref{4.1}).

\vspace{3mm}

We add a final remark. Tsukerman and Veomett \cite{TV17}, in their Section 3, aim at an extension of Cauchy's formula to moment vectors, but in an irritating way. Their notion $z_n(K_u)$ has different meanings in the formulation of Theorem 3 and in its proof. According to the definitions given on page 927, $z_n(K_u)= \int_{K_u} x \, \Ha^{n-1}(\D x)$, where $K_u$ is a part of the boundary of $K$. But this contradicts the third displayed formula on page 928, which should read
$$ D_{[o,u]}(z_{n+1})(K)= \lim_{\epsilon\to 0^+} \frac{z_{n+1}((K+\epsilon[o,u])\setminus K)}{\epsilon} = (n+1)z(K,\dots,K,[o,u]),$$
and this is different from $z_n(K_u)$. With the notations used above, we have
$$ (n+1)z(K,\dots,K,[o,u])= \int_{K|u^\perp} y_u^+(K,x)\,\Ha^{n-1}(\D x).$$
Therefore, with corrections, the argument on p. 928 of \cite{TV17} can be interpreted to yield the special case $f(x)=x$ of formula (\ref{4.1}).

\bigskip

\noindent
\textbf{Acknowledgements.}
D. Hug was supported by DFG research grant HU 1874/5-1 (SPP 2265).

\noindent Authors' addresses:\\[2mm]
Daniel Hug\\Karlsruhe Institute of Technology (KIT)\\
Department of Mathematics\\D-76128 Karlsruhe,
Germany\\
E-mail: daniel.hug@kit.edu\\[5mm]
Rolf Schneider\\Mathematisches Institut\\
Albert-Ludwigs-Universit{\"a}t\\
D-79104 Freiburg i.~Br., Germany\\
E-mail: rolf.schneider@math.uni-freiburg.de
\vfill

\end{document}